\documentclass{amsart}
\usepackage{amsmath}
\usepackage{amsopn}
\usepackage{amssymb}
\usepackage{latexsym}
\usepackage{amsfonts}
\usepackage{amsthm}
\usepackage[all]{xy}
\usepackage{amssymb,amsmath,amscd,amsthm,epsfig, tikz}

\usepackage[enableskew]{youngtab}

\newcommand{\D}{{\rm Des}}

\newcommand{\SYT}{{\rm {SYT}}}

\newcommand{\NN}{\mathbb{N}}

\newtheorem{thm}{Theorem}[section]

\newtheorem{corollary}[thm]{Corollary}

\newtheorem{question}[thm]{Question}
\theoremstyle{definition}

\newtheorem{exa}[thm]{Example}
\newtheorem{defn}[thm]{Definition}

\newtheorem{conj}[thm]{Conjecture}

\newcommand{\een}{\end{enumerate}}
\newcommand{\blem}{\begin{lem}}
\newcommand{\elem}{\end{lem}}
\newcommand{\bcl}{\begin{cla}}
\newcommand{\ecl}{\end{cla}}
\newcommand{\ethm}{\end{thm}}
\newcommand{\bpr}{\begin{pro}}
\newcommand{\epr}{\end{pro}}
\newcommand{\bco}{\begin{cor}}
\newcommand{\eco}{\end{cor}}
\newcommand{\bcon}{\begin{conj}}
\newcommand{\econ}{\end{conj}}
\newcommand{\bde}{\begin{defn}}
\newcommand{\ede}{\end{defn}}
\newcommand{\bex}{\begin{exa}}
\newcommand{\eexa}{\end{exa}}
\newcommand{\bobs}{\begin{obs}}
\newcommand{\eobs}{\end{obs}}
\newcommand{\bexe}{\begin{exe}}
\newcommand{\eexe}{\end{exe}}

\begin{document}

\title{A note on the number of $k$-roots in $S_n$}
\bibliographystyle{acm}

\author{Yuval Roichman}
\address{Department of Mathematics\\
Bar-Ilan University\\
52900 Ramat-Gan\\
Israel} \email{yuvalr@math.biu.ac.il}

\begin{abstract}
The number of $k$-roots of an arbitrary permutation is expressed
as an alternating sum of $\mu$-unimodal $k$-roots of the identity
permutation.
\end{abstract}

\date{submitted: Sep 15, '13;\ revised: Sep 15, '14}

\maketitle

\section{A Combinatorial Identity}

\subsection{Outline}
$\mu$-unimodality, which was introduced in computations of
Iwahori-Hecke algebra characters~\cite{HLR, Ra2, Ro2}, was applied
most recently to prove conjectures of Regev regarding induced
characters~\cite{ER} and of Shareshian and Wachs regarding
Stanley's chromatic symmetric function~\cite{Ath}. In this note it
will be shown that the number of $k$-roots of a permutation of
cycle type $\mu$ is equal to an alternating sum of $\mu$-unimodal
$k$-roots of the identity permutation.

\subsection{$\mu$-unimodal permutations}

Let $\mu=(\mu_1,\dots,\mu_t)$ be a partition of $n$ with $t$
nonzero parts. Denote
$$
\mu_{(0)}:=0
$$
$$
\mu_{(i)}:=\sum\limits_{j=1}^i \mu_i \qquad (1\le i\le t)
$$
and
\begin{equation}\label{S-mu}
S(\mu):=(\mu_{(1)},\dots,\mu_{(t)}).
\end{equation}

A permutation $\pi \in S_n$ is {\em $\mu$-unimodal} if for every
$0\le i< t$ there exist $0\le \i \le \mu_{i+1}$ such that
$$
\pi(\mu_{(i)}+1)>\pi(\mu_{(i)}+2)>\cdots > \pi(\mu_{(i)}+\i)<
\pi(\mu_{(i)}+\i+1)<\cdots< \pi(\mu_{(i+1)}).
$$
Denote the set of $\mu$-unimodal permutations in $S_n$ by $U_\mu$.

\medskip

For example, let $\mu=(\mu_1,\mu_2,\mu_3)=(4,3,1)$ then
$S(\mu)=(\mu_{(1)},\mu_{(2)},\mu_{(3)})=(4,7,8)$. The permutations
$53687142$ and $35687412$ are $\mu$-unimodal but $53867142$ and
$53681742$ are not.

Note that $U_{(1,\dots,1)}=S_n$.

\subsection{$k$-roots in $S_n$}

For $n\ge 1$ and $k\ge 0$ denote
$$
I^k_n:=\{\pi\in S_n:\ \pi^k=1\}
$$
the set of $k$-roots of the identity permutation in $S_n$.

\medskip

\begin{thm}\label{conjecture1}
For every $n\ge 1$, $k\ge 0$, partition $\mu\vdash n$ and $\pi\in
S_n$ of cycle type $\mu$ the following holds:
\begin{equation}\label{eq1}
\#\{\sigma\in S_n:\ \sigma^k=\pi\}=\sum\limits_{\sigma\in
I^k_n\cap U_\mu} (-1)^{|\D(\sigma)\setminus S(\mu)|}.
\end{equation}
\end{thm}

It follows that the set of $k$-roots of the identity permutation
is a fine set in the sense of~\cite{AR-matrices}. The case $k=2$
follows from~\cite[Prop.1.5]{APR}. Note that the proof there does
not apply to a general $k$.

\section{Proof of Theorem~\ref{conjecture1}}

\subsection{Induced representations}

For every $n\ge 1$ and $k\ge 0$ let $\theta^{k,n}:S_n
\longrightarrow \NN \cup\{0\}$ be the enumerator of $k$-roots of a
permutation $\pi$ in $S_n$
$$
\theta^{(k,n)}(\pi):=\#\{\sigma\in S_n:\ \sigma^k=\pi\}.
$$
Clearly, $\theta^{(k,n)}$ is a class function. By a classical
result of Frobenius and Schur $\theta^{(2,n)}$ is not virtual, see
e.g.~\cite[\S 4]{Isaacs}. It was conjectured by Kerber and proved
by Scharf~\cite{Scharf} that for every $k\ge 0$, $\theta^{(k,n)}$
is a non virtual character.

\medskip

Let $Z_\lambda$ be the centralizer of a permutation of cycle type
$\lambda$ in $S_n$. $Z_\lambda$ is isomorphic to the direct
product $\times_{i=1}^n C_i\wr S_{k_i}$, where $k_i$ is the
multiplicity of the part $i$ in $\lambda$. Denote by $\rho_i$ the
one dimensional representation of $C_i\wr S_{k_i}$ indexed by the
$i$-tuple of partitions $(\emptyset, (k_i), \emptyset,
\dots,\emptyset)$. Let
$$\rho^\lambda:=\bigotimes_{i=1}^n \rho_i
$$
a one-dimensional representation of $Z_\lambda$ and
$$
\psi^\lambda=\rho^\lambda\uparrow_{Z_\lambda}^{S_n}.
$$
the corresponding induced $S_n$-representation.

Denote by $\phi^{k,n}$ the  representation whose character is
$\theta^{(k,n)}$. The following theorem implies that $\phi^{k,n}$
is not virtual.

\begin{thm}\label{thm2}~\cite{Scharf}
For every $n\ge 1$ and $k\ge 0$
$$
\phi^{k,n}=\bigoplus_{\lambda\vdash n\atop \text{all parts divide
$k$} }\psi^\lambda.
$$
\end{thm}

See also~\cite[Cor. 5.2]{Thibon} and~\cite[Ex. 7.69(c)]{ECII}.
Note that letting $k=2$ gives Inglis-Richardson-Saxl's well known
construction of a Gelfand model for $S_n$~\cite{IRS}.

\subsection{Descents over conjugacy classes}

Let $C_\lambda$ be the conjugacy class of cycle type $\lambda$ in
$S_n$ and $\SYT(\nu)$ be the set of all standard Young tableaux of
shape $\nu$. Denote the multiplicity of the Specht module $S^\nu$
in $\psi^\lambda$ by $m(\nu,\lambda)$. The following is a
reformulation of~\cite[Thm. 2.1]{GR}, see also~\cite{JR}.

\begin{thm}\label{conjecture4}
For every $\lambda\vdash n$
\begin{equation}\label{eq2}
\sum\limits_{\pi\in C_\lambda} {\bf x}^{\D(\pi)}=
\sum\limits_{\nu\vdash n} m(\nu,\lambda) \sum\limits_{T\in
\SYT(\nu)} {\bf x}^{\D(T)} .
\end{equation}
\end{thm}

\begin{proof}
Denote by $L_\lambda$ the 
image of $\psi^\lambda$ under the Frobenius characteristic map.
For an explicit description of this symmetric function see
e.g.~\cite[Ex. 7.89]{ECII}. For $J\subseteq [n-1]$ let $z_J$ be
the skew Schur function which corresponds to the zigzag skew shape
with down steps on positions which belong to $J$. For example, in
the French notation, $J = \{1,4,5\}\subseteq[7]$ corresponds to
the shape
\[
\young(:1,:234,:::5,:::678) .
\]
By~\cite[Thm. 2.1]{GR}, the coefficient of ${\bf x}^J$ in the left
hand side of Equation (\ref{eq2}), which is the number of
permutations of cycle type $\lambda$ and descent set $J$, is equal
to $\langle L_\lambda, z_J\rangle$.

Now
$$
\langle L_\lambda, z_J \rangle= \langle L_\lambda, \sum_{\nu
\vdash n} \langle s_\nu, z_J\rangle  s_\nu \rangle = \sum_{\nu
\vdash n} \langle L_\lambda, s_\nu\rangle \langle s_\nu,
z_J\rangle = \sum_{\nu \vdash n} m(\nu,\lambda)\langle s_\nu,
z_J\rangle  .
$$
Since $\langle s_\nu, z_J\rangle$ is equal to number of SYT of
shape $\nu$ and descent set $J$~\cite[Thm. 7]{Gessel} (see
also~\cite[Thm. 4.1]{ABR}), this is equal to the coefficient of
${\bf x}^J$ in the right hand side of Equation (\ref{eq2}).
\end{proof}

\bigskip

\begin{corollary}\label{conjecture3}
For every partition $\mu\vdash n$ the value of $\psi^\lambda$ at a
permutation of cycle type $\mu$ is
\begin{equation}
\psi^\lambda_\mu=\sum\limits_{\sigma\in C_\lambda\cap U_\mu}
(-1)^{|\D(\sigma)\setminus S(\mu)|}.
\end{equation}
\end{corollary}

\begin{proof}
For partitions $\mu$ and $\nu$ of $n$ let $\chi^\nu_\mu$ be the
character value of the Specht module $S^\nu$ on a conjugacy class
of cycle type $\mu$. A standard Young tableau $T$ of size $n$ is
$\mu$-unimodal if $\D(T)\setminus S(\mu)$ is a disjoint union of
intervals of the form $[\mu_{(i)}+1,\mu_{(i)}+\i]$ for some $0\le
\i< \mu_{i+1}$. For example, the SYT
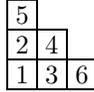
\[
\young(5,24,136)
\]
is $(3,3)$-unimodal but not $(4,2)$-unimodal.
\\
By~\cite[Theorem 4]{Ro2}~\cite{Ra2},
\[
\chi^\nu_\mu = \sum\limits_{T\in \SYT(\nu)\cap \SYT_\mu}
(-1)^{|\D(T)\setminus S(\mu)|},
\]
where $\SYT(\nu)$ is the set of all SYT of shape $\nu$ and
$\SYT_\mu$ is the set of $\mu$-unimodal SYT of size $n$.

Combining this with Theorem~\ref{conjecture4} gives
\begin{eqnarray*}
\psi^\lambda_\mu\ \ \ = \ \ \ \sum\limits_{\nu\vdash n}
m(\nu,\lambda) \chi^\nu_\mu & = & \sum\limits_{\nu\vdash n}
m(\nu,\lambda) \sum\limits_{T\in \SYT(\nu)\cap \SYT_\mu}
(-1)^{|\D(T)\setminus
S(\mu)|} \\
& = & \sum\limits_{\sigma\in C_\lambda\cap U_\mu}
(-1)^{|\D(\sigma)\setminus S(\mu)|}.
\end{eqnarray*}
\end{proof}

\medskip

\subsection{Conclusion}

By Theorem~\ref{thm2} together with Corollary~\ref{conjecture3},
for every $\pi\in S_n$ of cycle type $\mu$
\begin{eqnarray*}
\#\{\sigma\in S_n:\ \sigma^k=\pi\} \ \ \ = \ \ \
\theta^{(k,n)}(\pi) & = & \sum\limits_{\lambda\vdash n\atop
\text{all parts divide $k$}
}\psi^\lambda_\mu \\
 =\sum\limits_{\lambda\vdash n\atop \text{all
parts divide $k$}}\sum\limits_{\sigma\in C_\lambda\cap U_\mu}
(-1)^{|\D(\sigma)\setminus S(\mu)|} & = & \sum\limits_{\sigma\in
I^k_n\cap U_\mu} (-1)^{|\D(\sigma)\setminus S(\mu)|},
\end{eqnarray*}
completing the proof of Theorem~\ref{conjecture1}. \qed

\section{Remarks and questions}

It is desired to prove Theorem~\ref{conjecture1} via
generalizations of the explicit combinatorial construction of
Gelfand models described in~\cite{APR}.

\begin{question} Find a ``simple" $S_n$-linear action on a basis of $\phi^{k,n}$ indexed by $I^k_n$,
which will imply the character formula given in the right hand
side of Equation (\ref{eq1}).
\end{question}

Another desired approach to prove Theorem~\ref{conjecture1} is
purely combinatorial.

\begin{question}
Define, for any given partition $\mu$ of $n$, an involution on the
set of $k$-roots of the identity permutation, which changes the
parity of $\D(\cdot)\setminus S(\mu)$ on non-fixed points, such
that the cardinality of the fixed point set is equal to the LHS of
Equation (\ref{eq1}).
\end{question}

\begin{question} Prove Theorem~\ref{conjecture4} by constructing a map from
$C_\lambda$ to SYT of size $n$, under which for every $\nu\vdash
n$ and $T\in \SYT(\nu)$ the cardinality of the preimage of $T$ is
exactly $m(\nu,\lambda)$.
\end{question}

Note that for $\lambda=(2^k, 1^{n-2k})$, $0\le k \le n/2$, the RSK
map satisfies this property.

\bigskip

Finally, a natural objective is to extend the setting of the
current note to other finite groups. Complex reflection groups are
of special interest.

\begin{question}
Generalize Theorem~\ref{conjecture1} to other Coxeter and complex
reflection groups.
\end{question}

This question is intimately related to the problem of
characterizing the finite groups, for which the character
$\theta^{(k,n)}$ is non-virtual. For wreath products
see~\cite{Scharf_thesis}.

\bigskip

\noindent{\bf Acknowledgements:} Thanks to Ron Adin for useful
discussions and to Michael Schein and the anonymous referees for
helpful comments and references.

\end{document}